\documentclass[reqno, 12pt]{amsart}
\usepackage{amssymb,amsfonts}
\usepackage{hyperref} 
\usepackage{enumerate}
\usepackage{mathrsfs}
\usepackage[a4paper, left=2.5cm, right=2.5cm, top=2.5cm, bottom=2.5cm]{geometry}
\usepackage{setspace}

\theoremstyle{plain} 
\newtheorem{theorem}{Theorem}[section]
\newtheorem{lemma}[theorem]{Lemma}
\newtheorem{proposition}[theorem]{Proposition}

\theoremstyle{definition}

\newtheorem{remark}[theorem]{Remark}

\numberwithin{theorem}{section}

\newcommand{\R}{\mathbb{R}}
\newcommand{\C}{\mathbb{C}}
\newcommand{\Z}{\mathbb{Z}}

\renewcommand{\mod}{\operatorname{mod}}

\newcommand{\bb}[1]{\left(#1\right)}

\allowdisplaybreaks

\title{  On  sums of  finite subsets of the primes }

\author{Genheng Zhao}
\address{Academy of Mathematics and Systems Science, Chinese Academy of Sciences, Beijing 100190, China}
\email{zhaogenheng@amss.ac.cn}

\date{}

\begin{document}
\maketitle
\begin{abstract}
  Let   $A\subset [1,x]$ be a non-empty set of primes   with  $|A|= \alpha x(\log x)^{-1}$.   We prove that  there  exist absolute constants  $c_1,c_2>0$ such that, as $x$ gets sufficiently large,  we have  $|A+A|\geq c_1(\log x)(\log \log 3\alpha^{-1})^{-1}|A|$ if  $\alpha \geq c_2(\log x)^{-1/2}\log \log x$   and otherwise $|A+A|\geq c_1(\log x) (\log 2\alpha^{-1})^{-1}|A|$.  	\end{abstract}

\section{Introduction and main results}

Let $A$  be  a finite set of  integers.  It has been shown in many aspects that the size of  $A+A=\{a_1+a_2:a_1,a_2\in A\}$  conveys much structural information of $A$. For instance, Freiman's famous theorem \cite{Freiman} states  that if $|A+A|$ is  comparable to $|A|$, then  $A$ must be contained in some generalized arithmetical  progression of  proper dimension and size. Conversely, without  such structure, $|A+A|$ is supposed to be considerably  larger than  $|A|$. An interesting example is the set of primes. 

The binary Goldbach conjecture asserts that every   sufficiently large even integer  is a sum of two odd primes. Though this full conjecture is still out of reach by current methods, it is  proved in \cite{Lu} that for sufficiently large  $x$,  there are at most  $O(x^{0.879})$ exceptions among even integers in $[2,2x]$. Let $A$ be the set of primes in $[1,x]$. This means  $$|A+A|\gtrsim  (\log x)|A|.$$

It is believed that similar results will hold for subsets of the primes, provided the relative density is not too small. To quantify this, we suppose $A\subset[1,x] $ is a  set of  primes with $|A|= \alpha x(\log x)^{-1} $. Note first that by prime number theorem we have $\alpha \leq1+o_x(1)$.  

In  \cite{Ramare-Ruzsa}, Ramar\'e and Ruzsa considered relative dense subsets of the primes.   They showed that  for any fixed $\alpha>0$, as  $x\gtrsim_{\alpha} 1$, we have
\begin{equation}\label{eq: maki}
 	|A+A|\gtrsim \frac{\log x}{\log \log 3\alpha^{-1}}|A|,
 \end{equation}
  where the implicit constant was later made  explicit and sharp in the work of Matom\"aki \cite{Matomaki}.
  When $\alpha$ is not fixed, in particular when $\alpha=o_x(1)$ which corresponds to the case $A$ is  relative sparse,  Cui-Li-Xue \cite{Cui-Li-Xue} proved that for any $c_0>0$, if  $\alpha \geq (\log x)^{-c_0}$ then as $x\gtrsim 1$ we have 
\begin{equation}\label{eq: their-result}
|A+A|\gtrsim_{c_0} \frac{\log x}{\log \log x}|A|.
\end{equation}
Due to the presence of Siegel-Walfisz theorem in their method, the implicit constant here is ineffective.  

Observe that there is an obvious  gap between these two results, that is, \eqref{eq: their-result} does not recover \eqref{eq: maki}   if we assume $A$ is relative dense. To remedy this, we couple ideas from \cite{Ramare-Ruzsa} and \cite{Cui-Li-Xue} to derive the following theorem. 

\begin{theorem}\label{thm: main}  Let $A\subset [1,x]$ be a  set of primes   with  $|A|= \alpha x (\log x)^{-1}$.  As $x\gtrsim 1$, there exists an absolute constant $c>0$ such that for  $\alpha \geq c(\log x)^{-1/2}\log \log x$, we have
$$|A+A|\gtrsim \frac{\log x}{\log  \log 3\alpha^{-1}}|A|.$$
\end{theorem}

Since our method does not invoke Siegel-Walfisz theorem, the above result is effective. In fact, this  bound is also sharp in the sense that if we take $A=\{p\leq x:p\equiv a(\mod q)\}$ with $q=\prod_{p<z}p$, $z\leq \log \log x$  and $(a,q)=1$, then $\alpha \asymp \phi(q)^{-1}$, $z\asymp \log q$  and thus
$$|A+A|\lesssim  \frac{\phi(q)}{q}(\log x) |A|\asymp  \frac{\log x}{ \log z}|A|\asymp \frac{\log x}{\log  \log q} |A|\asymp \frac{\log x}{\log \log 3\alpha^{-1}}|A|. $$

For sparser $A$, we have instead the following bound, which is effective, slightly weaker than \eqref{eq: maki} but  stronger than \eqref{eq: their-result}. \begin{theorem}\label{thm:second}
Let $A\subset [1,x]$ be a non-empty set of primes   with  $|A|=\alpha x (\log x)^{-1}$.  As $x\gtrsim 1$,  we have	$$|A+A|\gtrsim \frac{\log x}{\log 2\alpha^{-1}}|A|.$$
\end{theorem}

\begin{remark} The above theorem is trivial if $\alpha \leq x^{-\delta }$ for some constant $\delta >0$. For these extremely sparse sets,   we can deduce from  \cite{Sanders} that there is some $c\in (0,1/4)$ such that     $$ |A+A|\geq  e^{\bb{\log \log |A|}^c}|A|,$$
provided  $|A|\geq e^{(\log \log x)^2}$ and $x\gtrsim 1$. This bound, however, is too weak to recover \eqref{eq: their-result}.  	 
\end{remark}

 To prove Theorem \ref{thm: main}, we will follow \cite{Cui-Li-Xue} to replace  the primes by   Granville's model  for primes to transfer our problem to a local one,  and then follow \cite{Ramare-Ruzsa} to solve the local problem. For  a detailed discussion of this probabilistic model for primes, see \cite{Granville}.  The main advantage  of our method takes place in the first step, where we use Brun's sieve coupled with several  analytic results   instead of Selberg's sieve  in \cite{Ramare-Ruzsa}. We also rely on the restriction  theorem for primes, in particular the sharp one proved by Ramar\'e \cite{Ramare2}, based on which we can derive a  restriction theorem for sparse subsets of the primes. And then Theorem \ref{thm:second} is a simple consequence. 
  
The rest of this paper is organized as follows. In Section 2 we fix some notations. In Section 3 we introduce Granville's model for primes. In Section 4 we introduce the sharp  restriction theorem and prove Theorem \ref{thm:second}. In section 5 we prove  Theorem \ref{thm: main}.
\section{Notation}
$\C:$ the set of complex numbers.

$\R:$ the set of real numbers.

$\Z:$ the set of complex numbers.

$\mathbb{P}: $ the set of prime numbers.

$\Z_q: $ the quotient group $\Z/q\Z$.

$\Z_q^*:$ invertible elements in $\Z_q$. 

$(q_1,q_2):$ the greatest common divisor of $q_1$ and $q_2$.

The asymptotic notations $\lesssim$, $\gtrsim $, $\asymp$ and $O()$ are used in a standard way. Besides, we will alway use $p$ to denote a prime and $1_S$ to denote the truth value of a statement $S$, that is,  $1_S=1$ if $S$ is true and otherwise $1_S=0$.

\section{Granville's model}
Let $x$ be a sufficiently large real number. Let $w\in [e^2, \frac{1}{2}\log  x]$ and  $W=\prod_{p<w} p $. These parameters will be in force throughout the rest of this paper. For  $n\leq x$,  Granville's model for $1_{n\in \mathbb{P}}\log n$ is defined by  $1_{(n,W)=1}\frac{W}{\phi(W)}$, where $\phi$ is Euler's totient function. We next define their Fourier transforms  by 
\begin{equation}\label{eq: s}
	S(\theta)=\sum_{p\leq x} (\log p)e(p\theta ),
\end{equation}
and \begin{equation}\label {eq: scr}
	G(\theta)=\frac{W}{\phi(W)}\sum_{\substack{n\leq x}}1_{(n,W)=1}e(n\theta),
\end{equation}
where we used the abbreviation $e\bb{\theta }=e^{2\pi i\theta }$.
	Our main task of this section is to verify  the idea
 that $G(\theta)$ is a good  approximation to $S(\theta)$, provided $w$ is large enough. 
 
 \begin{proposition}\label{prop: granville}As $x\gtrsim 1$, we have $$\sup_{\theta\in \R/\Z}|S(\theta)-G(\theta)|\lesssim \frac{\log \log w}{\sqrt{w}}x.$$
\end{proposition}
The proof of the above proposition is based on  circle method with respect to the parameter $Q=e^{\frac{1}{3} (\log x)^{1/3}}$. We now 
identify $\R/\Z$ with $(Qx^{-1},1+Qx^{-1}]$ and  define the major arcs by 
$$ \mathfrak{M}= \bigcup_{1\leq q\leq Q}\bigcup_{\substack{a=1\\ (a,q)=1}}^q \mathfrak{M}(a,q)=\{\theta\in  \R:|\theta-a/q|\leq Qx^{-1}q^{-1}\}.$$ Then  minor arcs are simply defined by $\mathfrak{m}=(Qx^{-1},1+Qx^{-1}]\setminus \mathfrak{M} $.  For convenience, we set $$r(n)=1_{n\in \mathbb{P}}\log n-1_{(n,W)=1}\frac{W}{\phi(W)}$$ and $R(\theta)=S(\theta)-G(\theta)$.

The estimate on minor arcs is quite standard. Let $\theta \in \mathfrak {m}$. 
By Dirichlet approximation theorem there exists a co-prime pair $a,q$ with $1\leq q\leq xQ^{-1}$ such that $|\theta-a/q|\leq Qx^{-1}q^{-1}$. Notice that when  $q\leq Q$ and $x>2Q^2$, we  have  $$\frac{1}{q}-\frac{Q}{xq}\leq \frac{1}{Q}-\frac{x}{Q}<\frac{x}{Q},$$ 
which means  $1\leq a\leq q$ and  then contradicts  to the fact $\theta\in \mathfrak{m}$. As $x\gtrsim 1$, we see that  $x>2Q^2$ holds and  thus $q\in (Q,xQ^{-1}]$.   Combing  \cite[Lemma 2.2]{Cui-Li-Xue}, \cite[Lemma 2.3]{Cui-Li-Xue} and the fact $$\frac{W}{\phi(W)}=\prod_{p<w}\bb{1-\frac{1}{p}}\lesssim \log w,$$ 
this leads to
$$ |R(\theta)|\lesssim  xQ^{-1/2}(\log x)^{O(1)}\lesssim xe^{-\frac{1}{12}(\log x)^{1/3}}.$$

\begin{lemma}\label{lem: minor}
 As $x\gtrsim 1$, we have $$\sup_ {\theta \in \mathfrak{m} }|R(\theta)|\lesssim \frac{\log \log w}{\sqrt{w}}x.$$
\end{lemma}
To proceed on major arcs, we need to estimate
$$\sum_{\substack{n\leq y\\ n\equiv l(\mod q)}}r(n)=\sum_{\substack{p\leq y\\ p\equiv l(\mod q)}}\log p-\frac{W}{\phi(W)}\sum_{\substack{n\leq y\\ n\equiv l(\mod q)}}1_{(n,W)=1}$$
for $1\leq l\leq q\leq Q$ and $y\in [\sqrt{x},x]$.    For the prime part, we  use two different types of the prime number theorem for arithmetic progressions to show the following.  
\begin{lemma} \label{lem: prime-average}Let  $1\leq l\leq q\leq Q$ and $y\in [\sqrt{x},x]$. As $x\gtrsim 1$, we have
$$\sum_{\substack{p\leq y\\ p\equiv l(\mod q)}}\log p=\frac{1_{(l,q)=1}}{\phi(q)} \bb{y-1_{q\in [w,Q]}\chi_1(l) \frac{y^{\beta_1}}{\beta_1}}+O(ye^{-(\log x)^{1/3} }),$$
where $\beta_1\in [1/2,1]$ is the  Landau-Siegel   zero with respect to  the Dirichlet character $ \chi_1 (\mod q)$. \end{lemma}
\begin{proof} 
	When $q\in [w,Q]$, using  \cite[Theorem 8.29]{Tenenenbaum} we have for  $y\ge 2$  and some absolute $c>0$ that 
	$$\sum_{\substack{p\leq y\\ p\equiv l(\mod q)}}\log p=\frac{1_{(l,q)=1}}{\phi(q)} \bb{y-\chi_1(l) \frac{y^{\beta_1}}{\beta_1}}+O(ye^{-c(\log y)^{1/2} }).$$
	As $x\gtrsim 1$, the error term is $O(ye^{-(\log x)^{1/3}})$ for $y\in [\sqrt{x},x]$. When $q<w$ and thus $q\leq \log x$, we instead use  \cite[Theorem 8.31]{Tenenenbaum}  to obtain for $y\geq 3$ and some absolute $c'>0$ that  
	$$ \sum_{\substack{p\leq y\\ p\equiv l(\mod q)}}\log p=\frac{1_{(l,q)=1}}{\phi(q)} y+O(ye^{-c'(\log y)^{1/2}(\log \log y)^{-2} }).$$
	As $x\gtrsim 1$, this error term is also $O(ye^{-(\log x)^{1/3}})$ for $y\in [\sqrt{x},x]$.
\end{proof}
As for the part of Granville's model, we appeal to the simplest form of Brun's sieve that for $r\geq 1$,
\begin{equation}
\label{eq: Brun}
\sum_{\substack{d\mid (n,W)\\ \omega(d)\leq 2r-1}}\mu(d)\leq1_{(n,W)=1}\leq \sum_{\substack{d\mid (n,W)\\ \omega(d)\leq 2r}} 	\mu(d),
\end{equation}
where $\omega(d)$ denotes the number of distinct prime divisors of $d$.

\begin{lemma}\label{lem: granville-average}Let  $1\leq l\leq q\leq Q$ and $y\in [\sqrt{x},x]$. Write $q=q_1q_2$ with $q_1=(q,W)$. As $x\gtrsim 1$, we have
$$ \frac{W}{\phi(W)}\sum_{\substack{p\leq y\\ p\equiv l(\mod q)}}1_{(n,W)=1}= \frac{1_{(l,q_1)=1}}{\phi(q_1)q_2}  y+O(ye^{-(\log x)^{1/3}}).$$
\end{lemma}
\begin{proof}We first   give an estimate on 
$$ \sum_{\substack{n\leq y\\ n\equiv l(\mod q)}} 1_{(n,W)=1}.$$
Notice   that by \eqref{eq: Brun}, we have
$$\sum_{\substack{d\mid W\\ \omega(d)\leq 2r-1}} 	\mu(d)\sum_{\substack{md\leq x\\ md\equiv l(\mod q)}}1\leq \sum_{\substack{n\leq y\\ n\equiv l(\mod q)}} 1_{(n,W)=1}\leq  \sum_{\substack{d\mid W\\ \omega(d)\leq 2r}} 	\mu(d)\sum_{\substack{md\leq x\\ md\equiv l(\mod q)}}1.
$$
It is clear that
$$\sum_{\substack{md\leq y\\ md\equiv l(\mod q)}}1= \rho (d)\frac{y}{q}+O(1),
 $$
where $\rho(d)=(q,d)/d$ if $(q,d)\mid l$ and  otherwise $\rho(d)=0$.  We thus have 
$$ \sum_{\substack{n\leq y\\ n\equiv l(\mod q)}} 1_{(n,W)=1}=  \frac{y}{q}\sum_{d\mid W}\mu(d)\rho(d) +O\bb{\sum_{\substack{d\mid  W\\ \omega(d)\geq 2r}}\rho(d)\frac{y}{q}}+O(w^{2r}).$$
Let $2r$  be the  nearest  even number to $\sqrt{\log y}$. It then follows that 
$$\sum_{d\mid W}\mu(d)\rho(d)= \prod_{p<w}(1-\rho(p))=1_{(l,q,W)=1}\prod_{\substack{p<w \\ p\nmid q}}\bb{1-\frac{1}{p}},$$
\begin{equation*}
	\begin{split}
		\sum_{\substack{d\mid  W\\ \omega(d)\geq 2r}}\rho(d)\frac{y}{q}\leq y\sum_{\substack{d\mid  W\\ \omega(d)\geq 2r}}\frac{1}{d}&\leq  y\sum_{k\geq 2r}\frac{1}{k!}\bb{\sum_{p<w}\frac{1}{p}}^{k}\\ &\lesssim y \sum_{k\geq 2r}\bb{C\frac{\log \log w}{k}}^{k}
		\\ &\lesssim y \sum_{k\geq 2r} e^{-k}
		\\&\lesssim ye^{-\sqrt{\log y}}.
	\end{split}
\end{equation*}

and 
$$w^{2r}=y^{2r\frac{\log w}{\log y }}\leq y^{O\bb{\frac{\log  w}{\sqrt{\log y}}}}\leq \sqrt{y},$$
as $x\gtrsim1 $.
As a result, we obtain
$$  \sum_{\substack{n\leq y\\ n\equiv l(\mod q)}}1_{(n,W)=1}= 1_{(l,q,W)=1} \frac{y}{q}\prod_{\substack{p<w \\ p\nmid q}}\bb{1-\frac{1}{p}}+O(ye^{-\sqrt{\log y}}).$$
Multiplying both sides by $W/\phi(W)$ and using $y\in [\sqrt{x},x]$ we arrive at
$$ \frac{W}{\phi(W)}\sum_{\substack{n\leq y\\ n\equiv l(\mod q)}}1_{(n,W)=1}=1_{(l,q,W)=1}\frac{y}{q} \prod_{\substack{p \mid (q,W)}}\bb{1-\frac{1}{p}}^{-1}+O(ye^{-(\log x)^{1/3}}),$$
which is essentially the desired result.
\end{proof}

Now we can give an estimate for $R(\theta)$ on major arcs.

\begin{lemma}\label{lem: major}As $x\gtrsim 1$, we have
$$ \sup_{\theta \in \mathfrak{M}}|R(\theta)|\lesssim  \frac{\log \log w}{\sqrt{w}}$$  
	
\end{lemma}

\begin{proof}
We now fix a major arc $\mathfrak{M}(a,q)$ with $1\leq a\leq q\leq Q$ and $(a,q)=1$. Let $z=\theta-a/q$. Observe that 
$$R(\theta)=\sum_{\substack{l=1}}^q e\bb{la/q}\sum_{\substack{n\leq x\\ n\equiv l(\mod q)}} r(n) e(zn).$$

When  $q<w$ and thus $q\mid W$, as $x\gtrsim 1$ we can apply  Lemma \ref{lem: prime-average} and Lemma  \ref{lem: granville-average} to deduce for $1\leq l\leq q$ and $y \in [\sqrt{x},x]$ that
$$ \sum_{\substack{n\leq y\\ n\equiv l(\mod q)}} r(n)=O(ye^{-(\log x)^{1/3}}),$$
 Integrating by parts this gives
$$ \sum_{\substack{n\leq x\\ n\equiv l(\mod q)}} r(n)e(zn)\lesssim (1+x|z|)xe^{-(\log x)^{1/3}}+\sqrt{x}\log x.$$
 Hence by $|z|\leq Qx^{-1}$ and $w\leq \frac{1}{2}\log x$,
$$ |R(\theta)|\lesssim Q^2xe^{-(\log x)^{1/3}}=xe^{-\frac{1}{3}(\log x)^{1/3}}\lesssim\frac{\log \log w}{\sqrt{w}} x $$

Otherwise, when $q\in [w,Q]$, we have instead
\begin{equation}
	\begin{split}
		 \sum_{\substack{n\leq y\\n \equiv l(\mod q)}} r(n)=&\frac{1_{(l,q)=1}}{\phi(q)} \bb{y-\chi_1(l) \frac{y^{\beta_1}}{\beta_1}}-\frac{1_{(l,q_1)=1}}{\phi(q_1)q_2}y\\&+O(ye^{-(\log x)^{1/3}}),
	\end{split}
\end{equation}
where $\chi_1,\beta_1$ are given  by Lemma \ref{lem: prime-average} and $q_1,q_2$ are given by Lemma \ref{lem: granville-average}.
As before, integrating by parts, this soon gives $R(\theta)=R_1(\theta)-R_2(\theta)-R_3(\theta)+O(xe^{-\frac{1}{3}(\log x)^{1/3}})$, with 
$$R_1(\theta)=\frac{1}{\phi(q)}\bb{\sum_{\substack{l=1\\ (l,q)=1}}^qe(la/q)}\int_0^x e(zu)du,$$
$$ R_{2}(\theta)=\frac{1}{\phi(q)}\bb{\sum_{l=1}^q \chi_1(l) e(la/q)}\int_0^x u^{\beta_1-1}e(zu)du,$$
and $$ R_{3}(\theta)=\frac{1}{\phi(q_1)q_2}\bb{\sum_{\substack{l=1\\(l,q_1)=1}}^q  e(la/q)}\int_0^x e(zu)du.$$  
It is well-known that  the Ramanujan sum 
$$\sum_{\substack{l=1\\(l,q)=1}}^qe(la/q)=\mu(q)$$
and the Gauss sum
$$ \left|\sum_{l=1}^q \chi_1(l) e(la/q)\right |\leq \sqrt{q}.$$
Writing $l=l_2q_2+l_1$,  we have
$$\sum_{\substack{l=1\\(l,q_1)=1}}^q  e(la/q)  = \sum_{l_2=0}^{q_2-1}\sum_{\substack{l_1=1\\(l_1,q_1)=1}}^{q-1}e(l_2a/q_1+l_1a/q)= \sum_{l_2=0}^{q_2-1}e(l_2a/q)\sum_{\substack{l_1=1\\(l_1,q_1)=1}}^{q-1}e(l_1a/q_1)=1_{q_2=1}\mu(q_1).$$
We thus  obtain  $$|R(\theta)|\lesssim  \bb{\frac{\sqrt{q}}{\phi(q)}+e^{-\frac{1}{3}(\log x)^{1/3}}}x\lesssim \frac{\log \log w}{\sqrt{w}}x.$$
	
	Since $\mathfrak{M}(a,q)$ is chosen arbitrarily,   we conclude the proof.

\end{proof}

\begin{proof}[Proof of Proposition 3.1] 	 Apply Lemma \ref{lem: major} and Lemma \ref{lem: minor}. The theorem follows readily since $(Qx^{-1},1+Qx^{-1}]=\mathfrak{M}\cup \mathfrak{m}$.
\end{proof}

\begin{remark}Assuming that  there is no Landau-Siegel zero,  the above argument  can give a stronger bound 
$$|R(\theta)|\lesssim \frac{\log \log w}{w}x,$$since there is no Gauss sum.
This bound will further  imply Theorem \ref{thm: main} for a wider  range $\alpha\geq c (\log x)^{-1}\log \log x$, which is the limit of this method.
	
\end{remark}

\section{Sharp Restriction theorem }

Let $A\subset [1,x]$ be a non-empty set of primes with $|A|=\alpha x(\log x)^{-1}$.  From Bourgain's result \cite[Equation (4.39)]{Bourgain}, one can deduce that  for any $q>2$,
$$\int_{0}^1 |\sum_{p\in A} e(p\theta)|^qd\theta \lesssim_q |A|^{q-1} \frac{\alpha^{1-q/2} }{\log x}.$$
Observe that, compared to Plancherel  identity where the bound $|A|^{q-1}$ is available, we gain an extra factor $\alpha^{\frac{2-q}{2}}(\log x)^{-1}$.  However, this factor makes sense only when $\alpha\geq  (\log x)^{-\frac{2}{q-2}}$. To treat sparser subsets of the primes,  we invoke   Ramar\'e's  work  \cite{Ramare2}.

\begin{proposition}\label{prop: restriction}Let $A\subset [1,x]$ be a non-empty set of primes with $|A|=\alpha x(\log x)^{-1}$.  As $x\gtrsim 1$, we have for $q>2$ that
$$\int_{0}^1 |\sum_{p\in A} e(p\theta)|^qd\theta \lesssim_q |A|^{q-1} \frac{\log 2\alpha^{-1} }{\log x}.$$
\end{proposition}
\begin{proof} Without loss of generality we can restrict $q\in (2,3]$. As $x\gtrsim 1$,   we can deduce from  \cite[Corollary 1.3]{Ramare2} that
$$ \int_{0}^1 |\sum_{p\in A} e(p\theta)|^qd\theta \lesssim  |A|^{q-1} \frac{\alpha^{1-q/2} }{(q-2)\log x}.$$
Let $q'-2=(q-2)(10\log 2\alpha^{-1})^{-1}<q-2$ and apply the above to  $q'$. We have
$$\int_{0}^1 |\sum_{p\in A} e(p\theta)|^qd\theta\leq |A|^{q-q'} \int_{0}^1 |\sum_{p\in A} e(p\theta)|^{q'}d\theta \lesssim_{q}|A|^{q-1} \frac{\log 2\alpha^{-1}}{\log x}. $$
\end{proof}

 Inserting the  weight $\log p$ into \cite[Corollary 1.3]{Ramare2}, we also obtain a weighted version of the above.

\begin{proposition}\label{prop: weighted-restriction }Let $A\subset [1,x]$ be a non-empty set of primes with $\sum_{p\in A} \log p=\alpha  x$. Then as $x\gtrsim 1$, we have for $q>2$ that
$$\int_{0}^1 |\sum_{p\in A} (\log p)e(p\theta)|^qd\theta \lesssim_q (\alpha x)^{q-1} \log 2\alpha^{-1} .$$
\end{proposition}

Since the argument of \cite{Ramare2} is sieve-theoretic,  Proposition \ref{prop: restriction} and Proposition \ref{prop: weighted-restriction } are both effective. To see the sharpness we can simply take $A=\{2\}$.  Finally, to conclude this section, we give the short proof of Theorem \ref{thm:second}.

\begin{proof}[Proof of Theorem 1.2]Let $A\subset[1,x]$ be a non-empty set of integers with $|A|=\alpha x(\log x)^{-1}$. By Cauchy-Schwarz inequality  we have
\begin{equation*}
\begin{split}
	|A|^4=\bb{\sum_{n\in A+A}\sum_{\substack{n=p_1+p_2\\p_1,p_2\in A}}1}^2&\leq |A+A| 	\sum_{\substack{p_1+p_2=p_3+p_4\\p_1,p_2,p_3,p_4\in A}}1\\&=|A+A|\int_0^1|\sum_{p\in A}e(p\theta)|^4d\theta
\end{split}
\end{equation*}
Meanwhile, as $x\gtrsim 1$, by Proposition \ref{prop: restriction}  we have
$$\int_{0}^1|\sum_{p\in A}e(p\theta)|^4d\theta\lesssim |A|^3\frac{\log 2\alpha^{-1}}{\log x}.$$
\end{proof}

\section{Proof of   Theorem 1} 
Before proving Theorem \ref{thm: main}, we show a local result which provides the crucial saving in our argument. 
\begin{lemma}\label{lem: local} Suppose $L_1,L_2\subset \Z_W$ with $|L_1|,|L_2|\leq \beta W$ and $\beta\in (0,1]$. Then
$$\sum_{\substack{i+j\in \Z_W^*\\i \in L_1,j\in L_2}}1\lesssim |L_1||L_2|\cdot \frac{\log \log 3\beta^{-1}}{\log w}$$ 
\end{lemma}

\begin{proof} The case $\log 3\beta^{-1}\geq  \ w$ is trivial so we assume $\log 3\beta^{-1}<   w$. Using \cite[Theorem 3]{Ramare-Ruzsa}, we have $$\sum_{\substack{i+j\in \Z_W^*\\i\in L_1,j\in L_2}}1 \lesssim |L_1||L_2|  \exp\bb{-\sum_{\log 3\beta^{-1}\leq p< w}\frac{1}{p}+3(\log 3\beta^{-1})\sum_{\log 3\beta^{-1}\leq p< w}\frac{1}{p^2}}.$$
Since
$$\exp\bb{-\sum_{\log 3\beta^{-1}\leq p< w}\frac{1}{p}}\lesssim \frac{\log \log 3\beta^{-1}}{\log w} $$
and 
$$ \exp\bb{(3\log 3\beta^{-1})\sum_{\log 3\beta^{-1}\leq p< w}\frac{1}{p^2}}\lesssim 1,$$
we conclude the proof.
\end{proof}

\begin{proof}[Proof of Theorem 1.1]   Let $A\subset [1,x]$ be a non-empty set of primes and  define $$S_A(\theta)=\sum_{\substack{p\in A }}(\log p)e(p\theta).$$
Suppose first that $|A|=\alpha x (\log x)^{-1}$ with $\alpha\geq (\log x)^{-1}$. As $x\gtrsim 1$, we see$$S_A(0)\geq \sum_{p\leq \frac{\alpha}{2} x}\log p\gtrsim \alpha x.$$
Note that by Cauchy-Schwarz inequality
\begin{equation}\label{eq: energy}
\begin{split}
S_A(0)^4=\bb{\sum_{n\in A+A}\sum_{\substack{n=p_1+p_2\\p_1,p_2\in A}}\log p}^2&\leq |A+A| 	\sum_{\substack{p_1+p_2=p_3+p_4\\p_1,p_2,p_3,p_4\in A}}(\log p_1)(\log p_2)(\log p_3)(\log p_4)
\\&=|A+A|\int_0^1|S_A(\theta)|^4d\theta.
\end{split}
\end{equation}
We now use Granville's model to estimate the energy integral
$$E=\int_0^1|S_A(\theta)|^4d\theta.$$ Let $S(\theta)$ be as in  \eqref{eq: s} and $G(\theta)$ be as in \eqref{eq: scr}. Notice that 

$$\sum_{\substack{p_1+p_2=p_3+p_4\\p_1,p_2,p_3,p_4\in A}}(\log p_1)(\log p_2)(\log p_3)(\log p_4)\leq \sum_{\substack{p+p_1=p_2+p_3\\p\leq x\\p_1,p_2,p_3\in A\\}}(\log p)(\log p_2)(\log p_3)(\log p_4).$$
We have
$$E\leq \int_0^1S(\theta)S_A(-\theta)|S_A(\theta)|^2d\theta.$$
By Proposition \ref{prop: granville}, this gives
$$ E\lesssim \frac{\log \log w}{\sqrt{w}}x\int_0^1|S_A(\theta)|^3d\theta +\int_0^1G(\theta)S_A(-\theta)|S_A(\theta)|^2d\theta.$$
Denote the above inequality by $E\lesssim  E_1+E_2$. As $x\gtrsim 1$, we  deduce from Proposition  \ref{prop: weighted-restriction } that 
$$E_1\lesssim \frac{\log \log w}{\sqrt{w}}x^3 \alpha ^2\log 2\alpha^{-1}. $$
To estimate $E_2$, we notice that
$$E_2\leq \frac{W}{\phi(W)}(\log x)^3\sum_{p_1,p_2,p_3\in A} 1_{(W,p_2+p_3-p_1)=1}.$$
For each $j\in \Z_W$,  denote $A_j=\{p\in A:p\equiv j(\mod W)\}$. Then
$$\sum_{p_1,p_2,p_3\in A} 1_{(W,p_2+p_3-p_1)=1}=\sum_{p\in A}\sum_{\substack{i+j\in \Z_W^*\\ i,j \in \Z_W}}|A_i||A_{j+p}|\leq \frac{\alpha x}{\log x} \max_{p\in A}\sum_{\substack{i+j\in \Z_W^*\\ i,j \in \Z_W}}|A_i||A_{j+p}|.$$
Meanwhile, we let $|A_j|=\alpha_j x(\phi(W)\log x)^{-1}$. Thus
$$E_2\leq \alpha x^3\frac{W}{\phi(W)^3} \max_{p\in A}\sum_{\substack{i+j\in \Z_W^*\\ i,j \in \Z_W}}\alpha_i\alpha_{j+p}$$
Recall that $W\leq \prod_{p< \frac{1}{2}\log x}p\lesssim x^{3/4}$. Then by Brun-Titchmarsh  theorem (see \cite[Theorem 4.16]{Tenenenbaum}) we have $\alpha_j\in [0,10]$, as $x\gtrsim 1$. Now we have  the constraint
\begin{equation}\label{eq: constraint}
	\sum_{j\in \Z_W}\alpha_j=\alpha \phi(W),\quad \alpha_{j}\in [0,10] \text{ for }j\in  \Z_W.
\end{equation}
Suppose that $\{ \alpha_j'\}_{j\in \Z_W}$ and $\{\alpha''_j\}_{j\in \Z_W}$ maximize the bilinear form
\begin{equation}\label{eq: sum}
	 \Phi(\alpha',\alpha'')=\sum_{\substack{i+j\in \Z_W^*\\ i,j \in \Z_W}}\alpha_i'\alpha''_j
\end{equation}
 with both of them  subject to \eqref{eq: constraint}. We claim that we can assume there is at most one $j\in \Z_W$ such that $\ \alpha_j' \not \in \{0,10\}$ and the same for $\alpha''$. Indeed, if there exist $ \alpha_{j_1}', \alpha_{j_2}'\in (0,10)$ with $j_1\neq j_2$, then for any $ S_2\geq S_1\geq 0$, we have
 $$\alpha_{j_1}'S_1+\alpha_{j_2}'S_2\leq \begin{cases}
 	(\alpha_{j_1}'+\alpha_{j_2}'-10)S_1+10S_2,&  \alpha_{j_1}'+ \alpha_{j_2}'\geq 10,\\0\cdot S_1+ (\alpha_{j_1}'+\alpha_{j_2}')S_2, & \alpha_{j_1}'+ \alpha_{j_2}'< 10.
 \end{cases}$$ 
 This means we can replace $\alpha_{j_1}',\alpha_{j_2}'$ by $\alpha_{j_1}'+\alpha_{j_2}'-10,10$ or $0,\alpha_{j_1}'+\alpha_{j_2}'$to let  $\Phi(\alpha',\alpha'')$ increase  with $\alpha'$ still subject to $\eqref{eq: constraint}$.

  Now we have
  \begin{equation*}
  \begin{split}
  	\Phi(\alpha',\alpha'')&\leq 20\alpha \phi(W)+ 100\max_{\substack{L_1,L_2\subset \Z_W\\|L_1|,|L_2|\leq \frac{\alpha}{10}\phi(W) }} \sum_{\substack{i+j\in \Z_{W}^*\\i\in L_1,j\in L_2}}1\\&\lesssim \alpha \phi(W)+\alpha^2\phi(W)^2\frac{\log \log (3\alpha^{-1}\log w)}{\log w},
  \end{split}
  \end{equation*}
  where in the last step we used Lemma \ref{lem: local}. It then follows
  \begin{equation*}
  	\begin{split}
  	E_2&\lesssim \alpha^2x^3\frac{W}{\phi(W)^2}+ \alpha^3x^3 \frac{\log \log (3\alpha^{-1}\log w)}{\log w }\frac{W}{\phi(W)}\\&\lesssim \alpha ^2x^3e^{-\frac{1}{2}w}+\alpha^3x^3\log \log (3\alpha^{-1}\log w)  	\end{split}
  \end{equation*}
  and thus
  $$E\lesssim E_1+E_2\lesssim \alpha^3x^3\bb{\alpha^{-1}e^{-\frac{1}{2}w}+\log \log (3\alpha^{-1}\log w)+\frac{\log \log w}{\sqrt{w}}\alpha^{-1}\log 2\alpha^{-1}}.$$  
  Set $w=10e^2(\alpha^{-1}\log 2\alpha^{-1})^2\geq e^2$. Let $c_1>0$ be such that   $\alpha^{-1}\leq c_1 (\log x)^{1/2}(\log \log x)^{-1} $ implies $w\leq \frac{1}{2}\log x$. Then   $$ E\lesssim \alpha^3x^3 (1+\log \log 3\alpha^{-1}+\log \log w)\lesssim \alpha^3x^3 \log \log 3\alpha^{-1}.$$
 Combing this with \eqref{eq: energy} we conclude the proof with $c=c_1^{-1}$.
\end{proof}

\end{document}